\documentclass{article}

\usepackage{amsmath,yhmath,amsfonts,amssymb,amsthm,bbding}
\usepackage{mathrsfs}

\newcommand{\dist}{{\operatorname{dist}}}
\newcommand{\Alex}{{\operatorname{Alex}}}

\newcommand{\counte}{theorem}

\newtheorem{defn}[\counte]{\bf Definition}

\newtheorem{prop}[\counte]{\bf Proposition}
\newtheorem{lemma}[\counte]{\bf Lemma}

\newtheorem{coro}[\counte]{\bf Corollary}

\newtheorem{remark}[\counte]{\bf Remark}

\allowdisplaybreaks

\numberwithin{equation}{section}

\renewcommand{\thefootnote}{\fnsymbol{footnote}}

\begin{document}

\renewcommand{\thefootnote}{\arabic{footnote}}

\centerline{\bf\Large Quasi-convex subsets and the farthest direction }
\centerline{\bf\Large in Alexandrov spaces with lower curvature bound \footnote{Supported by NSFC 11971057 and BNSF Z190003. \hfill{$\,$}}}

\vskip5mm

\centerline{Xiaole Su, Hongwei Sun, Yusheng Wang\footnote{The
corresponding author (E-mail: wyusheng@bnu.edu.cn). \hfill{$\,$}}}

\vskip6mm

\noindent{\bf Abstract.} This paper aims to give a further study on quasi-convex subsets
in Alexandrov spaces with lower curvature bound which are introduced in \cite{SSW}.
We first provide new insights on quasi-convex subsets (Theorem A and Corollary C), and then as applications of them we explore more properties of quasi-convex subsets.
\vskip1mm


\vskip1mm


\vskip6mm

\setcounter{section}{-1}


\section{Introduction}

Finite dimensional Alexandrov spaces with lower curvature bound
can be viewed as a generalization of Riemannian manifolds with lower sectional curvature bound
(\cite{BGP}). Compared to Riemannian manifolds,  Alexandrov spaces might have some singularities,
so that some important subsets appear as some kinds of analogues of totally geodesic submanifolds,
such as convex subsets, extremal subsets and quasi-geodesics ([PP1-2]). Recently,
such a kind of subsets named quasi-convex subsets has been introduced (\cite{SSW}).
They include not only all convex subsets without boundary and extremal subsets but also more other subsets,
such as fixed point sets of isometries on Alexandrov spaces.
Moreover, all shortest pathes in a quasi-convex subset are quasi-geodesics.

To show the definition of quasi-convex subsets, we first make some conventions on notations.

\noindent$\bullet$\ \ $\text{Alex}(k)$: the set of complete and finite dimensional Alexandrov spaces with
curvature $\geqslant k$.

\noindent$\bullet$\ \ $\Bbb S^2_k$: the complete and simply connected space form of dimension 2 and curvature $k$.

\noindent$\bullet$\ \ $|pq|$, $[pq]$: the distance, a minimal geodesic (i.e. shortest path) between $p$ and $q$.

\noindent$\bullet$\ \ $\uparrow_p^q$: the direction from $p$ to $q$ for a given $[pq]$.

\noindent$\bullet$\ \ $\Uparrow_p^q$: the set of all directions from $p$ to $q$.

\noindent For a point $p\in X\in \text{Alex}(k)$, we denote by $\Sigma_pX$ the space of directions of $X$ at $p$ which
belongs to $\text{Alex}(1)$ ([BGP]). In $\Sigma_pX$, $\Uparrow_p^q$ is a closed subset, and
$|\Uparrow_p^q\Uparrow_p^r|$ is the distance between $\Uparrow_p^q$ and $\Uparrow_p^r$. And to $p,q,r\in X$, we associate $\tilde p, \tilde q, \tilde r\in \mathbb S^2_k$
with $|\tilde p\tilde q|=|pq|, |\tilde p\tilde r|=|pr|$ and $|\tilde q\tilde r|=|qr|$ \footnote{If $k>0$ and
$|pq|+|pr|+|qr|=\frac{2\pi}{\sqrt k}$, it is necessary to add a condition that $\tilde p, \tilde q$ and $\tilde r$ lie in a geodesic of length $\frac{\pi}{\sqrt k}$.},
and then we denote by $\tilde\angle_k qpr$ the angle at $\tilde p$ of the triangle
$\triangle \tilde p\tilde q\tilde r$.

\begin{defn}\label{def0.1}{\rm In an $X\in \text{Alex}(k)$, a closed subset $F$ is called to be {\it quasi-convex}
if the following condition is satisfied: if the distance function to $q\not\in F$ restricted to $F$, $\dist_q|_F$, attains a minimum at $p\in F$, then for all $r\in F\setminus\{p\}$
\begin{equation} \label{eqn0.1}
|\Uparrow_p^q\Uparrow_p^r|\leqslant\frac{\pi}{2} \text{\ \ (or equivalently, $\tilde\angle_k qpr\leqslant\frac{\pi}{2}$)}.
\end{equation}
Here, we make a convention that both the empty set and a single point are quasi-convex in $X$.}
\end{defn}

By Toponogov's Theorem\footnote{For the theorem, one can refer to Section 3 in [BGP] (or Theorem 1.1 in \cite{SSW}).}, it is obvious that `$|\Uparrow_p^q\Uparrow_p^r|\leqslant\frac{\pi}{2}$' implies `$\tilde\angle_k qpr\leqslant\frac\pi2$', but not vice versa; however, they are equivalent to each other in the situation of Definition \ref{def0.1} (\cite{SSW}). Moreover,
due to the arbitrariness of $q$, (\ref{eqn0.1}) is in fact equivalent to $|\uparrow_p^q\Uparrow_p^r|\leqslant\frac{\pi}{2}$ for any $[pq]$.

\begin{remark}\label{rem0.2}{\rm In this paper, we say that $F$ is {\it extremal} in $X$ if (\ref{eqn0.1}) holds for all $r\in X\setminus\{p\}$ in Definition \ref{def0.1}. This coincides with the concept `extremal' in \cite{PP1} if $F$ contains at least two points (\cite{SSW}). If $F$ is the empty set or a single point and if $k=1$,
some extra conditions are added in \cite{PP1} for some kind of completeness (cf. \cite{SSW}).}
\end{remark}

\begin{remark}\label{rem0.3}{\rm In \cite{SSW}, locally quasi-convex subsets are also defined. In detail,
a subset $F$ in an $X\in \text{Alex}(k)$ is {\it locally quasi-convex}
if for any $x\in F$ there is a neighborhood $U_x$ of $x$ such that $U_x\cap F$ is closed and if
$\dist_q|_F$ with $q\not\in F$ attains a minimum at $p\in F\cap U_x$, then the corresponding
(\ref{eqn0.1}) holds for all $r\in F\cap U_x\setminus\{p\}$. In a complete Riemannian manifold, a closed and locally quasi-convex subset must be a totally geodesic submanifold.}
\end{remark}

It is obvious that the quasi-convexity (as well as the extremality) of $F$ is determined by the geometry at points realizing minimums of $\dist_q|_F$ with $q\notin F$. A natural question is what is the essential geometry to a general point of $F$.
The first result of this paper gives an answer to it.

\vskip2mm

\noindent {\bf Theorem A.}{\it \label{thmA}
Let $F$ be a closed subset in an $X\in \text{\rm Alex}(k)$. Then $F$ is quasi-convex in $X$ if and only if, for any two distinct points $p, r\in F$ and $\eta\in\Sigma_pX$, there is $\zeta\in\Uparrow_p^r$ and $\xi\in\Sigma_pF$ with $|\eta\xi|\leqslant\frac\pi2$ such that
\begin{equation} \label{eqn0.2}\cos|\eta\zeta|\geqslant\cos|\eta\xi|\cos|\zeta\xi|;
\end{equation}
where if $\Sigma_pF=\emptyset$  (i.e. $p$ is an isolated point of $F$), then (\ref{eqn0.2}) means that $\cos|\eta\zeta|\geqslant 0$. }

\vskip2mm

Note that, in Theorem A, if $\eta$ lies in $\Sigma_pF$, then we can let $\xi=\eta$.
And an alternative formulation of (\ref{eqn0.2}) is $\tilde\angle_1 \eta\xi\zeta\leqslant\frac{\pi}{2}$.

\begin{remark}\label{rem0.4}{\rm In Theorem A, according to the proof of Theorem A,
if $F$ is quasi-convex in $X$, we can in fact select $\xi$
with $|\eta\xi|\leqslant\frac\pi2$ to satisfy a bit stronger version of (\ref{eqn0.2}):
$$\cos|\eta\Uparrow_p^r|\geqslant\cos|\eta\xi|\cos|\Uparrow_p^r\xi|.$$}
\end{remark}

Restricted to `extremal' case, Theorem A can be formulated as follows, which can be seen almost obviously from the proof of Theorem A.

\vskip2mm

\noindent {\bf Corollary B.}{\it \label{croB}
Let $F$ be a closed subset in an $X\in \text{\rm Alex}(k)$. Then
$F$ is extremal in $X$ if and only if, for any $p\in F$ and any $\eta,\zeta\in\Sigma_pX$, there is $\xi\in\Sigma_pF$ with $|\eta\xi|\leqslant\frac\pi2$  such that $\cos|\eta\zeta|\geqslant\cos|\eta\xi|\cos|\zeta\xi|$ (which means that $\cos|\eta\zeta|\geqslant 0$ if $\Sigma_pF=\emptyset$).}

\vskip2mm

From Theorem A, we can derive another equivalent condition of quasi-convexity, which will be
very helpful to see some nice properties of quasi-convex subsets.

\vskip2mm

\noindent {\bf Corollary C.}{\it \label{corC}
Let $F$ be a closed subset in an $X\in \text{\rm Alex}(k)$. Then $F$ is quasi-convex in $X$ if and only if, for any two distinct points $p, r\in F$, if there is $\eta\in\Sigma_pX$ satisfying $|\eta\Uparrow_p^r|>\frac\pi2$, then the farthest direction to $\Uparrow_p^r$ in $\Sigma_pX$ belongs to $\Sigma_pF$. }

\vskip2mm

Note that the farthest direction $\xi$ to $\Uparrow_p^r$ in $\Sigma_pX$,
i.e. $|\xi\Uparrow_p^r|=\max\{|\nu\Uparrow_p^r| : \nu\in \Sigma_pX\}$, is unique if the maximum is bigger than $\frac\pi2$
(by Toponogov's Theorem).
In Corollary C, if $\Sigma_pF=\emptyset$, then ``if there is $\eta\in\Sigma_pX$ satisfying $|\eta\Uparrow_p^r|>\frac\pi2$, then ...'' means that $|\eta\Uparrow_p^r|\leqslant\frac\pi2$ for all $\eta\in\Sigma_pX$.
In Sections 2 and 4, we will present two equivalent versions of Corollary C (see Propositions \ref{prop2.1} and \ref{prop4.1}).

Similarly, for `extremal' case, Corollary C can be formulated as follows:
{\it In an $X\in \text{\rm Alex}(k)$, a closed subset $F$ is extremal if and only if,
for any $p\in F$ and $\zeta\in\Sigma_pX$, if there is $\eta\in\Sigma_pX$ satisfying $|\eta\zeta|>\frac\pi2$,
then  the farthest direction to $\zeta$ in $\Sigma_pX$ belongs to $\Sigma_pF$.}
The necessity of the extremality of $F$ here can be seen from \cite{PP1},
and the sufficiency has been given in \cite{Pet}.

\begin{remark}\label{rem0.5}{\rm Based on Remark \ref{rem0.3}, one can give the corresponding versions
of Theorem A and Corollary C (and Propositions \ref{prop2.1} and \ref{prop4.1}) for local quasi-convexity. }
\end{remark}

\begin{remark}\label{rem0.6}{\rm
Let $F$ be a quasi-convex subset in $X\in \text{\rm Alex}(k)$. \cite{SSW} has shown two important properties of $F$.
One is that $\Sigma_pF$ is also quasi-convex in $\Sigma_pX$ for any $p\in F$,
and the other is that a shortest path in $F$ is a quasi-geodesic in $X$. They can be proven
just from Definition \ref{def0.1} without involving the new viewpoints in Theorem A and Corollary C.}
\end{remark}

In the rest of the paper, we first give proofs of Theorem A and Corollary C in Sections 1 and 2.
Then, as applications of them, we will show several properties of quasi-convex subsets.
For instance, two points in a quasi-convex subset can be jointed with a curve in the subset if they are close sufficiently
(see Section 3), and the intersection of two quasi-convex subsets is also quasi-convex (see Section 4.2).
Moreover, based on Corollary C, we can illustrate quasi-convex subsets by gradient curves of distance functions (see Section 4.1). This will make it easy to judge some kinds of subsets to be quasi-convex (see Sections 4.2-4), such as the fixed point set of an isometry.


\section{Proofs of Theorem A and Corollary B}

The main goal of this section is to give a proof for Theorem A.
In the proof, we will use the following facts in Alexandrov geometry.

\begin{lemma}[\cite{BGP}]\label{lem1.1}
Let $X\in \text{\rm Alex}(k)$ and $p\in X$. Then for any small $\epsilon>0$
there is a neighborhood $U_\epsilon$ of $p$ such that, for any $[pq],[pr], [qr]\subset U_\epsilon$,
$$0\leqslant|\uparrow_p^q\uparrow_p^r|-\tilde\angle_k qpr<\epsilon,\ \
0\leqslant|\uparrow_q^p\uparrow_q^r|-\tilde\angle_k pqr<\epsilon\ \text{ and }\
0\leqslant|\uparrow_r^p\uparrow_r^q|-\tilde\angle_k prq<\epsilon.$$
\end{lemma}

\begin{lemma}[\cite{BGP}]\label{lem1.2}
Let $X\in\text{\rm Alex}(k)$, and let $p, r, q_i\in X$ with $q_i\to p$ as $i\to\infty$.
If there is $[pq_i]$ such that $\uparrow_p^{q_i}$ converges to $\eta\ (\in\Sigma_pX)$ as $i\to\infty$,
then
$$\lim_{i\to\infty}\frac{|q_ir|-|pr|}{|pq_i|}=-\cos|\Uparrow_p^{r}\eta|.$$
As a result, $\lim\limits_{i\to\infty}\tilde\angle_krpq_i=|\Uparrow_p^{r}\eta|$.
\end{lemma}

\begin{proof}[Proof of Theorem A]\

We first show the necessity of the quasi-convexity in the theorem.
Note that we can let $\xi=\eta$ if $\eta\in \Sigma_{p} F$,
so we can assume that $\eta\notin \Sigma_{p} F$. Then
there is $q_n\in X\setminus F$	such that $q_n \to p $ and $\uparrow_{p}^{q_n}\to \eta$ as $n\to \infty$.
Let $p_n\in F$ satisfy $|q_np_n|=|q_nF|$. Note that $p_n\to p$ as $n\to \infty$,
and we can complete the proof according to the following two cases.

Case 1: there is at least a subsequence of $\{p_n\}$ which belongs to $F\setminus\{p\}$.
In this case, there must be a subsequence
$\{p_{n_i}\}$ with $p_{n_i}\neq p$ such that $\uparrow_{p}^{p_{n_i}}$ converges to some $\xi\in\Sigma_pF$ as $i\to\infty$.
We claim that $|\eta\xi|\leqslant\frac\pi2$ and
\begin{equation} \label{eqn1.1}\cos|\eta\Uparrow_p^r|\geqslant\cos|\eta\xi|\cos|\Uparrow_p^r\xi|.
\end{equation}
This implies that there is $\zeta\in\Uparrow_p^r$ such that (\ref{eqn0.2}) holds. So, we just need to verify the claim in this case. For simpleness, we still denote by $\{p_n\}$ the subsequence $\{p_{n_i}\}$.
Since $|q_np_n|<|q_np|$ and $q_n, p_n\to p$ as $n\to\infty$,
by Lemma \ref{lem1.1} it is easy to see that $|\eta\xi|\leqslant\frac\pi2$. For (\ref{eqn1.1}),
due to the similarity, we only give a proof for the case where $k=0$. By the quasi-convexity of $F$, we have that
\begin{equation} \label{eqn1.2}
\tilde \angle_k  q_np_nr\leqslant  \frac{\pi}{2}.
\end{equation}
Consequently, by the Law of Cosine, we have that 		
$$|rp_n|^2+|q_np_n|^2\geqslant |q_nr|^2,$$
where
\begin{align*}
	|rp_n|^2 & =|rp|^2+|pp_n|^2-2|rp|\cdot |pp_n|\cdot \cos \tilde \angle_0 rpp_n, \\
	|q_np_n|^2  & =|q_np|^2+|pp_n|^2-2|q_np|\cdot |pp_n|\cdot \cos \tilde \angle_0 q_npp_n,\\
    |q_nr|^2  & =|q_np|^2+|pr|^2-2|q_np|\cdot |pr|\cdot \cos \tilde \angle_0 q_npr.
\end{align*}
It then follows that
\begin{align}
	\cos \tilde\angle_0 rpq_n & \geqslant -\frac{|pp_n|^2}{|pq_n|\cdot |pr|} +\frac{|pp_n|}{|pq_n|}\cos\tilde \angle_0 rpp_n+\frac{|pp_n|}{|pr|}\cos\tilde \angle_0 q_npp_n. \label{eqn1.3}
\end{align}
Since $q_n, p_n\to p$, $\uparrow_{p}^{q_n}\to \eta$ and $\uparrow_{p}^{p_n}\to \xi$ as $n\to\infty$, by Lemma \ref{lem1.2} we have that
\begin{equation} \label{eqn1.4}
\tilde \angle_0 rpq_n\to |\Uparrow_{p}^r\eta |\  \text{ and }\
\tilde \angle_0 rpp_n \to |\Uparrow_{p}^r\xi|\ \text{ as }\ n\to \infty.
\end{equation}	
Moreover, we make a subclaim:
\begin{equation} \label{eqn1.5}
\lim_{n\to\infty}\frac{|pp_n|}{|pq_n|}=\cos|\eta \xi|.
\end{equation}
Note that (\ref{eqn1.3})-(\ref{eqn1.5}) together with $\frac{|pp_n|}{|pr|}\to 0$ as $n\to \infty$ implies (\ref{eqn1.1}).
In order to see the subclaim, by Lemma \ref{lem1.1} it suffices to show that
there is $[p_nq_n]$ and $[p_np]$ such that $|\uparrow_{p_n}^{q_n}\uparrow_{p_n}^p|\to\frac\pi2$ as $n\to\infty$. By the quasi-convexity of $F$, there is $[p_nq_n]$, $[p_np]$ and $[p_np_m]$ with $m\neq n$
such that
$$|\uparrow_{p_n}^{q_n}\uparrow_{p_n}^p|\leqslant\frac\pi2 \text{ and }
|\uparrow_{p_n}^{q_n}\uparrow_{p_n}^{p_{m}}|\leqslant\frac\pi2.$$
On the other hand, since $p_n\to p$ and $\uparrow_{p}^{p_n}\to \xi$ as $n\to\infty$,
it is not hard to see that $|\uparrow_{p_n}^p\uparrow_{p_n}^{p_{m}}|\to \pi$
where $n\gg m$ and $m\to\infty$ (by Toponogov's Theorem), which implies
$$|\uparrow_{p_n}^{q_n}\uparrow_{p_n}^p|+|\uparrow_{p_n}^{q_n}\uparrow_{p_n}^{p_{m}}|\to \pi.$$ It therefore follows that
$|\uparrow_{p_n}^{q_n}\uparrow_{p_n}^p|\to\frac\pi2$ as $n\to\infty$ (i.e. the subclaim is verified, so is the claim).

Case 2: $p_n=p$ for all large $n$. In this case, the quasi-convexity of $F$ guarantees that $|\Uparrow_p^{q_n} \Uparrow_{p}^{r}|\leqslant\frac{\pi}{2}$, and thus $|\eta \Uparrow_{p}^{r}|\leqslant\frac{\pi}{2}$
(note that `$q_n\to p$ and $\uparrow_{p}^{q_n}\to \eta$' implies $\Uparrow_{p}^{q_n}\to \eta$, cf. \cite{BGP}).
I.e., there is $\zeta\in\Uparrow_p^r$  such that $\cos|\eta\zeta|\geqslant 0$, so the proof of the necessity is done if $\Sigma_p F=\emptyset$.
If $\Sigma_p F\neq\emptyset$, we claim that
	\begin{align*}
	& |\uparrow_p^{q_n}\xi|=\frac{\pi}{2} \ \text{ for  any } \ \xi\in \Sigma_p F,
	\end{align*}
which implies $|\eta\xi|=\frac{\pi}{2}$.
In fact, we have that $|\uparrow_p^{q_n}\xi|\geqslant\frac\pi2$ because $|q_np|=|q_nF|$ (by Lemma \ref{lem1.2}),
and meanwhile $|\uparrow_p^{q_n}\xi|\leqslant\frac\pi2$ (by the quasi-convexity of $F$).
Note that the claim right above implies that any $\xi\in \Sigma_p F$ satisfies (\ref{eqn0.2}), and thus the proof of the necessity is done.

\vskip2mm

Next, we will verify the sufficiency of the quasi-convexity in Theorem A. We argue by contradiction.
Suppose that $F$ is not quasi-convex in $X$. Then there exists $q\not\in F$ and $p,r\in F$ such that $|qp|=|qF|$, but for some $[qp]$ we have that
\begin{equation} \label{eqn1.6}
	|\uparrow_{p}^q\Uparrow_{p}^r|>\frac{\pi}{2}.
 \end{equation}
We let $\eta\triangleq \uparrow_{p}^q$. By the assumption, there is $\zeta\in\Uparrow_{p}^r$ and $\xi\in\Sigma_pF$ with $|\eta\xi|\leqslant\frac\pi2$ such that
$$\cos |\eta\zeta| \geqslant \cos|\eta\xi| \cdot \cos|\zeta\xi|.$$
Note that `$|\eta\zeta |>\frac\pi2$' (see (\ref{eqn1.6})) together with `$|\eta\xi|\leqslant\frac\pi2$' implies that $|\eta\xi|$ must be less than $\frac\pi2$. However, since $|qp|=|qF|$, by Lemma \ref{lem1.2} we have that $|\eta\xi|\geqslant\frac\pi2$,
a contradiction.
\end{proof}

By restricting the above proof to `extremal' case, we can easily derive Corollary B.

\begin{proof}[Proof of Corollary B]\

If $\zeta$ is the direction of some $[pr]$ at $p$ (here $r$ is not needed to lie in $F$), one just need to replace all
`quasi-convexity' with `extremality' in the proof of Theorem A.
Then for the case where $\zeta$ cannot be realized by a minimal geodesic, we can draw the conclusion by a limiting argument (note that there is $[pr_i]$ with $r_i\to p$ as $i\to\infty$ such that $\uparrow_p^{r_i}\to\zeta$).
\end{proof}

\begin{remark}\label{rem1.3}{\rm Similar to Corollary B, if $F$ is a nonempty quasi-convex subset in
$X\in \text{\rm Alex}(k)$, and if $\Sigma_pF\neq\emptyset$ with $p\in F$, then for any $\eta\in\Sigma_pX$ and $\zeta\in\Sigma_pF$ there is $\xi\in\Sigma_pF$ with $|\eta\xi|\leqslant\frac\pi2$  such that $\cos|\eta\zeta|\geqslant\cos|\eta\xi|\cos|\zeta\xi|$ (by Theorem A). However, the converse might not be true (as an example,
one can consider a submanifold in a Riemannian manifold which is not totally geodesic). }
\end{remark}


\section{Proof of Corollary C}

In this section, we will first prove Corollary C, and then show an alternative version of it.

\begin{proof} [Proof of Corollary C]\

We first verify the necessity of the quasi-convexity in the corollary.
Let $p,r\in F$. Since there is $\eta\in\Sigma_pX$ such that $|\eta\Uparrow_p^r|>\frac\pi2$,
the farthest direction $\xi_0$ to $\Uparrow_p^r$ in $\Sigma_pX$ is unique and $|\xi_0\Uparrow_p^r|>\frac\pi2$.
By Theorem A (and Remark \ref{rem0.4}, see (\ref{eqn1.1})), the quasi-convexity of $F$ implies that there is $\xi\in\Sigma_pF$ with $|\xi_0\xi|\leqslant\frac\pi2$  such that
$$\cos|\xi_0\Uparrow_p^r|\geqslant\cos|\xi_0\xi|\cos|\Uparrow_p^r\xi|.$$
Note that `$|\xi_0\Uparrow_p^r|>\frac\pi2$', `$|\xi_0\Uparrow_p^r|\geqslant|\xi\Uparrow_p^r|$' and `$|\xi_0\xi|\leqslant\frac\pi2$' imply that $\xi_0$ has to be equal to $\xi$, so it follows that $\xi_0\in\Sigma_pF$.

Next, we will verify the sufficiency, and argue by contradiction.
Suppose that $F$ is not quasi-convex in $X$. Then there exists $q\not\in F$ and $p,r\in F$ such that $|qp|=|qF|$, but for some $[qp]$ we have that
$$|\Uparrow_{p}^r\uparrow_{p}^q|>\frac{\pi}{2}.	$$
Hence, by the assumption, the farthest direction $\xi$ to $\Uparrow_p^r$ in $\Sigma_pX$
belongs to $\Sigma_pF$. We claim that $|\uparrow_{p}^q\xi|<\frac\pi2$, which
contradicts `$|qp|=|qF|$' (by Lemma \ref{lem1.2}).
In fact, if $|\uparrow_{p}^q\xi|\geqslant\frac\pi2$, by Toponogov's Theorem it is not hard to see that
$|\Uparrow_p^r\xi'|>|\Uparrow_p^r\xi|$ for $\xi'$ near $\xi$ in any $[\uparrow_{p}^q\xi]$
because $|\Uparrow_{p}^r\uparrow_{p}^q|>\frac{\pi}{2}$ and $|\Uparrow_p^r\xi|\geqslant|\Uparrow_{p}^r\uparrow_{p}^q|>\frac{\pi}{2}$; a contradiction
(because $\xi$ is the farthest direction to $\Uparrow_p^r$).
\end{proof}

We now provide an alternative version of Corollary C, which will be easy to be used
in next section. We formulate it in the following proposition.

\begin{prop}\label{prop2.1}
Let $F$ be a closed subset in an $X\in \text{\rm Alex}(k)$. Then $F$ is quasi-convex in $X$ if and only if, for any two distinct points $p, r\in F$, if $\eta \in \Sigma_p X$ satisfies $-\cos|\Uparrow_{p}^r \eta|>0$, then there exists  $\xi\in \Sigma_{p}F$ such that
\begin{equation} \label{eqn2.1}
-\cos |\Uparrow_{p}^r\xi | \geqslant -\cos|\Uparrow_{p}^r \eta|.
\end{equation}
Moreover, when $F$ is quasi-convex and $-\cos|\Uparrow_{p}^r \eta|>0$, the $\xi$ in (\ref{eqn2.1}) can be chosen to satisfy $\cos |\eta\xi|\geqslant-\cos|\Uparrow_{p}^r \eta|$; as a result, there is  $\xi'\in \Sigma_{p}F$ such that
\begin{equation} \label{eqn2.2}
\cos |\Uparrow_{p}^r\xi' | \geqslant -\cos|\Uparrow_{p}^r \eta|.
\end{equation}
\end{prop}

\begin{proof} Note that the first statement is an alternative formulation of Corollary C. So, we just need to verify the second one. Assume that $F$ is quasi-convex and $\eta \in \Sigma_p X$ satisfies $-\cos|\Uparrow_{p}^r \eta|>0$.
By Theorem A, there is $\zeta\in\Uparrow_p^r$ and $\xi\in\Sigma_pF$ with $|\eta\xi|\leqslant\frac\pi2$  such that
$$\cos|\eta\zeta|\geqslant\cos|\eta\xi|\cos|\zeta\xi|.$$
This implies that $\cos |\eta\xi|\geqslant -\cos|\Uparrow_{p}^r \eta|$ because $-\cos|\eta\zeta|\geqslant -\cos|\Uparrow_{p}^r \eta|>0$ and $\cos|\eta\xi|\geqslant0$.

We next show that there is  $\xi'\in \Sigma_{p}F$ such that (\ref{eqn2.2}) holds. Note that there
is a sequence of $r_i\in F$ such that $r_i\to p$ and $\Uparrow_p^{r_i}\to\xi$ as $i\to\infty$.
And for any $\zeta\in \Uparrow_p^r$, by replacing $\Uparrow_p^r$ and $\eta$ with $\Uparrow_p^{r_i}$ and $\zeta$ respectively,
we can conclude that there is  $\xi_i\in \Sigma_{p}F$ such that $\cos |\zeta\xi_i|\geqslant -\cos|\Uparrow_p^{r_i}\zeta|$,
which implies $\cos |\Uparrow_p^{r}\xi_i|\geqslant -\cos|\Uparrow_p^{r}\Uparrow_p^{r_i}|$.
Thereby, for the limit $\xi'$ of any converging subsequence of $\{\xi_i\}$ (note that $\Sigma_pF$ is a
closed subset in $\Sigma_pX$), we have that
$\cos |\Uparrow_p^{r}\xi'|\geqslant -\cos|\Uparrow_p^{r}\xi|\geqslant -\cos|\Uparrow_{p}^r \eta|$.
\end{proof}

\begin{remark}\label{rem2.2}{\rm By replacing $\Uparrow_{p}^r$ with any $\zeta\in\Sigma_pX$ in Proposition \ref{prop2.1},
we can get the corresponding version of the proposition for `extremal' case. Here, the existences of $\xi$ and $\xi'$ by the extremality of $F$ can be seen from Proposition 1.6 in \cite{PP1}. }
\end{remark}

\vskip2mm

In the rest of the paper, we will present some properties of quasi-convex subsets, as applications of
the idea of Theorem A (and its equivalent versions---Corollary C and Proposition 2.1).

\section{Connectedness of quasi-convex subsets}

It is known that the number of extremal subsets in a compact space of $\Alex(k)$ is finite (Proposition 3.6 in \cite{PP1}). Unfortunately, there is no such strong conclusion on quasi-convex subsets. For example, in a standard sphere,  there are infinitely many great circles each of which is quasi-convex. Nevertheless, we have
a weaker conclusion for quasi-convex subsets.

\begin{prop} \label{prop3.1}
In a compact space $X\in \Alex(k)$, any quasi-convex subset has a finite number of connected components.
\end{prop}

In fact, we have the following stronger conclusion than Proposition \ref{prop3.1} (which
corresponds to (2) of Corollary 3.2 in \cite{PP1} for `extremal' case).

\begin{prop} \label{prop3.2}
Let $X$ be a compact space in $\Alex(k)$, and let $F$ be a quasi-convex subset in $X$. Then there is $\epsilon>0$ (depending on $X$) such that, for any two distinct points $p,q\in F$ with $|pq|<\epsilon^2$,  there exists a curve in $F$ jointing $p$ and $q$ with length bounded from above by $\frac{|pq|}{\epsilon}$.
\end{prop}

Moreover, we can see the following property which corresponds to Proposition 3.3 in \cite{PP1} for `extremal' case.

\begin{prop}\label{prop3.3}
Let $F$ be a quasi-convex subset in an $X\in\Alex(k)$, and let $p\in  F$ and $\xi\in \Sigma_{p} F$. Then there exists a curve in $F$ starting from $p$ and tangent to the direction $\xi$.
\end{prop}

In proving Propositions 3.1 and 3.2, the following lemma (Lemma 3.1 in \cite{PP1}) is needed, which is some kind of essential geometry of Alexandrov spaces with lower curvature bound.

\begin{lemma}\label{lem3.4}
Let $X$ be a compact space in $\Alex(k)$. Then there is $\epsilon>0$ (depending on $X$) such that,
for any two distinct points $p, q \in X$ with $|p q|<\epsilon^{2}$, at least one of the following holds:
$$\max _{\eta \in \Sigma_{q} X}\operatorname{dist}_p'|_q(\eta)>\epsilon
\text{\ \ and }\ \max _{\eta \in \Sigma_{p} X}\operatorname{dist}_q'|_p(\eta)>\epsilon,$$
where $\operatorname{dist}_p'|_q(\eta)$ denotes the derivative of $\operatorname{dist}_p$ (the distance function to $p$) at $q$ along the direction $\eta$.
\end{lemma}

Note that $\dist_p'|_q(\eta)=-\cos|\Uparrow_{q}^p\eta|$ (by Lemma \ref{lem1.2}). Hence, by Proposition \ref{prop2.1} (see (\ref{eqn2.1}) and (\ref{eqn2.2})) we can easily see the following property.

\begin{lemma} \label{lem3.5}
In Lemma \ref{lem3.4}, if $p, q$ lie in a quasi-convex subset $F$ of $X$ additionally,  then
	$$\max _{\xi \in \Sigma_{q} F}\dist_p'|_q(\xi)>\epsilon \quad \text{and}\quad  \min _{\xi \in \Sigma_{q} F}\dist_p'|_q(\xi)<-\epsilon,$$
	or
	$$\max _{\xi \in \Sigma_{p} F}\dist_q'|_p(\xi)>\epsilon \quad \text{and}\quad  \min _{\xi \in \Sigma_{p} F }\dist_q'|_p(\xi)<-\epsilon.$$
\end{lemma}

Lemma \ref{lem3.5} corresponds to (1) of Corollary 3.2 in \cite{PP1} for `extremal' case, where $p$ and $q$ can lie in two distinct extremal subsets.

Since our proofs for Propositions 3.1 and 3.2 are imitations of their corresponding versions for `extremal' case in \cite{PP1}, we just provide rough proofs for them.

\begin{proof} [Proof of  Proposition \ref{prop3.1}]\

Let $F$ be a quasi-convex subset in $X$, and let $F_1$ and $F_2$ be two connected components  of $F$. It suffices to show that the distance between $F_1$ and $F_2$ is bigger than $\epsilon^2$,
where $\epsilon$ is the number associated to $X$ satisfying Lemma \ref{lem3.4}. If this is not true, then there is $p\in F_1$ and $q\in F_2$ such that $|pq|=|F_1F_2|<\epsilon^2$ (note that
$X$ is compact and $F$ is closed in $X$). It then follows that
$$\min_{\xi\in \Sigma_p F_1}\dist_q'|_p(\xi)\geqslant 0  \quad \text{and}\quad \min_{\xi\in \Sigma_q F_2}\dist_p'|_q(\xi)\geqslant 0,$$
which contradicts Lemma \ref{lem3.5}.
\end{proof}

\begin{proof} [Proof of  Proposition \ref{prop3.2}]\

Let $p,q$ be two distinct points in $F$. By Lemma \ref{lem3.5}, there is an $\epsilon$ such that if $|pq|<\epsilon^2$, then
$$
\min _{\xi \in \Sigma_{q} F} \operatorname{dist}_{p}^{\prime} |_{q}(\xi)<-\epsilon\ \text{ or } \
\min _{\xi \in \Sigma_{p} F} \operatorname{dist}_{q}^{\prime} |_{p}(\xi)<-\epsilon.
$$
Then, for each $n\in \Bbb N^+$,
it is not hard to see that there are two sequences of points $\{p_i\}_{i=1}^\infty$ and $\{q_i\}_{i=1}^\infty$ (depending on $n$)
in $F$ with
\begin{equation}\label{eqn3.1}
\text{either }\ \ p_i=p_{i-1} \text{ and } 0<|q_{i}q_{i-1}|<\frac1n\ \ \text{ or } \ \
q_{i}=q_{i-1} \text{ and } 0<|p_ip_{i-1}|<\frac1n
\end{equation}
(where $p_0=p$ and $q_0=q$) such that
\begin{equation}\label{eqn3.2}
|p_{i-1}q_{i-1}|-|p_iq_i|>\epsilon (|p_{i-1}p_i|+|q_iq_{i-1}|)
\end{equation}
and
\begin{equation}\label{eqn3.3}
|p_iq_i|\to 0 \text{ as } i\to\infty.
\end{equation}
Let $\{p_{i,n}, q_{i,n}\}_{i=0}^\infty$ denote the above two point sequences corresponding to each $n$.
Note that $$\sum_{i=1}^\infty(|p_{i-1,n}p_{i,n}|+|q_{i,n}q_{i-1,n}|)<\frac{|pq|}{\epsilon}$$ for all $n$.
Hence, as $n\to \infty$ and passing to a subsequence of $\{n\}$,
$\{p_{i,n}, q_{i,n}\}_{i=0}^\infty$ converges to a curve in $F$ jointing $p$ and $q$ with length $\leqslant\frac{|pq|}{\epsilon}$.
\end{proof}

\begin{remark}\label{rem3.6}{\rm  In the proof of Proposition \ref{prop3.2}, for general $\{p_i\}_{i=1}^\infty$ and $\{q_i\}_{i=1}^\infty$ satisfying (\ref{eqn3.1}) and (\ref{eqn3.2}), it might occur that
$p_i\to\bar p$ and $q_i\to\bar q$ as $i\to\infty$  with $\bar p\neq\bar q$. Note that $|\bar p\bar q|<|pq|<\epsilon^2$, so
similarly there is $p'$ and $q'$ in $F$ with either $p'=\bar p$ and $|q'\bar q|<\frac1n$ or
$q'=\bar q$ and $|p'\bar p|<\frac1n$ such that
$|\bar p\bar q|-|p'q'|>\epsilon (|\bar pp'|+|q'\bar q|)$. Hence, for sufficiently large $i$, we can reset
$p_i=p_{i-1}$ and $q_i=q'$ or $q_i=q_{i-1}$ and $p_i=p'$ so that the new $p_i$ and $q_i$ still satisfy (\ref{eqn3.1}) and (\ref{eqn3.2}). Such an idea enables us to find $\{p_i\}_{i=1}^\infty$ and $\{q_i\}_{i=1}^\infty$ satisfying (\ref{eqn3.1})-(\ref{eqn3.3}). }
\end{remark}

As for Proposition \ref{prop3.3}, we can almost directly copy the proof of Proposition 3.3 in \cite{PP1} (here the basis is Lemma \ref{lem3.5}
instead of Corollary 3.2 in \cite{PP1}). However, we would like to provide a proof for it via Corollary C without involving Lemmas \ref{lem3.5} and \ref{lem3.4}.

\begin{proof} [Proof of  Proposition \ref{prop3.3}]\

Since $\xi\in\Sigma_pF$, there is $\{p_i\}_{i=1}^\infty\subset F$ such that $p_i\to p$ and $\Uparrow_p^{p_i}\to\xi$ as $i\to\infty$. Note that, for sufficiently small $\delta>0$, there is $i_0$ such that
\begin{equation} \label{eqn3.4}
|\Uparrow_p^{p_i}\Uparrow_p^{p_{i_0}}|<\frac{\delta}3 \text{ for all } i>i_0.
\end{equation}
Moreover, by Lemma \ref{lem1.1}, we can assume that
\begin{equation} \label{eqn3.5}
|\uparrow_p^x\uparrow_p^y|-\tilde{\angle }_k x p y<\frac{\delta}{3} \text{ for any } [px], [py]\subset\overline{B_{p}(|pp_{i_0}|)}.
\end{equation}

In the rest of the proof, for any $i>i_0$ with $|pp_i|\ll|pp_{i_0}|$, we will first construct an arc-length parameterized curve $\alpha_i(t)|_{t\in[0,\frac{|p_ip_{i_0}|}{2}]}$ $\subset F$ with $\alpha_i(0)=p_i$ such that
$$|\Uparrow_p^{\alpha_i(t)}\xi|<\delta \text{ for all } t\in[0, \frac{|p_ip_{i_0}|}{2}].$$
And then we will show that, as $i\to\infty$, $\alpha_i(t)$ converges to a curve we want.

In order to construct $\alpha_i(t)$, we claim that, for any $\epsilon\ll |pp_i|$, there is $\{z_j\}_{j=1}^{N(\epsilon)}\subset F$ with $z_1=p_i$ and a constant $C$ such that
$$|z_jz_{j+1}|<\epsilon,\ \left|\sum_{j=1}^{N(\epsilon)-1}|z_jz_{j+1}|-\frac{|p_ip_{i_0}|}{2}\right|<\epsilon \text{ and } \tilde\angle_kz_{j+1}pp_{i_0}<\tilde\angle_kz_jpp_{i_0}+\frac{C\epsilon}{|pp_i|}|z_jz_{j+1}|.$$
Note that (\ref{eqn3.4}) implies that $\tilde\angle_kz_1pp_{i_0}<\frac\delta3$ (by Toponogov's Theorem) and
$|\Uparrow_p^{p_{i_0}}\xi|\leqslant\frac{\delta}3$. Then taking into account (\ref{eqn3.5}), we can see that
$\{z_j\}_{j=1}^{N(\epsilon)}$ converges to the desired $\alpha_i(t)$ as $\epsilon\to 0$.

We now verify the above claim. Since $\tilde\angle_kp_ipp_{i_0}<\frac{\delta}3$
and $\tilde\angle_kpp_{i_0}p_i\ll \frac{\delta}3$ (note that $|pp_i|\ll |p_ip_{i_0}|$), we can assume that $\tilde\angle_kpp_ip_{i_0}>\pi-\delta$, and thus by Toponogov's Theorem we have that
$$|\Uparrow_{p_i}^{p}\Uparrow_{p_i}^{p_{i_0}}|>\pi-\delta.$$
By Corollary C, the farthest direction $\xi_i$ to $\Uparrow_{p_i}^p$ in $\Sigma_{p_i}X$ belongs to $\Sigma_{p_i}F$.
By Lemma \ref{lem1.2} and Toponogov's Theorem, the `farthest' property of $\xi_i$ implies that for any $\eta\in\Uparrow_{p_i}^{p_{i_0}}$ there is  $\zeta\in\Uparrow_{p_i}^{p}$
such that $\tilde\angle_1\zeta\xi_i\eta\leqslant\frac\pi2$. It then follows that
$$
\cos |\Uparrow_{p_i}^p\Uparrow_{p_i}^{p_{i_0}}|\geqslant\cos |\zeta\eta|\geqslant \cos|\xi_i\zeta|\cdot\cos|\xi_i\eta|\geqslant \cos|\xi_i\zeta|\cdot\cos|\xi_i\Uparrow_{p_i}^{p_{i_0}}|
$$
(note that $|\xi_i\zeta|\geqslant |\xi_i\Uparrow_{p_i}^{p}|\geqslant |\Uparrow_{p_i}^{p}\Uparrow_{p_i}^{p_{i_0}}|>\pi-\delta$, and thus it holds that $|\xi_i\eta|<\frac\pi2$),
which implies that
\begin{equation}\label{eqn3.6}
\cos |\Uparrow_{p_i}^p\Uparrow_{p_i}^{p_{i_0}}|\geqslant -\cos|\xi_i\Uparrow_{p_i}^{p_{i_0}}|,
\text{ or equivalently, } |\Uparrow_{p_i}^p\Uparrow_{p_i}^{p_{i_0}}|+|\xi_i\Uparrow_{p_i}^{p_{i_0}}|\leqslant\pi.
\end{equation}
Denote by $z_1$ the point $p_i$, and note that $\xi_i\in\Sigma_{z_1}F$. Then for any $\epsilon\ll |pz_1|$, (\ref{eqn3.6}) together with Lemma \ref{lem1.2} implies that there is $z_2\in F$ such that $|z_1z_2|<\epsilon$ and
\begin{equation}\label{eqn3.7}
\tilde\angle_kpz_1p_{i_0}+\tilde\angle_kp_{i_0}z_1z_2<\pi+\epsilon.
\end{equation}
Let $\bar z_2,\bar p,\bar p_{i_0}\in \Bbb S_k^2$ satisfy $|\bar p\bar p_{i_0}|=|pp_{i_0}|$, $|\bar z_2\bar p_{i_0}|=|z_2p_{i_0}|$ and $|\bar p\bar z_2|=|pz_1|+|z_1z_2|$.
For a special case of (\ref{eqn3.7}) where $\tilde\angle_kpz_1p_{i_0}+\tilde\angle_kp_{i_0}z_1z_2\leqslant\pi$,
we notice that $\angle \bar z_2\bar p\bar p_{i_0}\leqslant\tilde\angle_kz_1pp_{i_0}\ (<\frac\delta3)$
by Alexandrov's lemma (Lemma 2.5 in \cite{BGP}),
which implies $\angle \bar p\bar z_2\bar p_{i_0}>\pi-\delta\ (>\frac\pi2)$. From `$\angle \bar p\bar z_2\bar p_{i_0}>\frac\pi2$',
it follows that $\tilde\angle_kz_2pp_{i_0}\leqslant\angle \bar z_2\bar p\bar p_{i_0}$ (note that $|pz_2|\leqslant|\bar p\bar z_2|$), so
\begin{equation}\label{eqn3.8}
\tilde\angle_kz_2pp_{i_0}<\tilde\angle_kz_1pp_{i_0} \text{ and } \tilde\angle_kpz_2p_{i_0}>\pi-\delta.
\end{equation}
In general, it is not so hard to conclude that
$\angle \bar z_2\bar p\bar p_{i_0}<\tilde\angle_kz_1pp_{i_0}+\frac{C\epsilon}{|pp_i|}|z_1z_2|$, where
$C$ is a constant depending only on $|pp_{i_0}|$.
Then we can similarly see that
\begin{equation}\label{eqn3.9}
\tilde\angle_kz_2pp_{i_0}\leqslant\angle \bar z_2\bar p\bar p_{i_0}<\tilde\angle_kz_1pp_{i_0}+\frac{C\epsilon}{|pp_i|}|z_1z_2| \text{ and } \tilde\angle_kpz_2p_{i_0}>\pi-\delta.
\end{equation}
Note that $|\Uparrow_{z_1}^p\xi_i|>\pi-\delta$, which implies that $|pz_2|>|pz_1|$ (by Lemma \ref{lem1.2}).
Then based on (\ref{eqn3.8}) and (\ref{eqn3.9}), we can similarly locate $z_j$ with $j\geqslant 3$ one by one.
Moreover, for a similar reason in Remark \ref{rem3.6}, there is an $N(\epsilon)$ such that
$\left|\sum_{j=1}^{N(\epsilon)-1}|z_jz_{j+1}|-\frac{|p_ip_{i_0}|}{2}\right|<\epsilon$. So far, the claim has been verified.

Note that  $\alpha_i(t)$ converges to a curve $\alpha(t)|_{t\in[0,\frac{|pp_{i_0}|}{2}]}$ with $\alpha(0)=p$ as $i\to\infty$ (here, there might be a need of passing to a subsequence, but in fact not; cf. Remark \ref{rem3.7} below). We just need to show that $\alpha(t)$ is tangent to $\xi$ at $p$.
In fact, for any integer $n>1$, there is $i_n\gg i_0$ such that
$|\Uparrow_p^{p_i}\Uparrow_p^{p_{i_n}}|<\frac{\delta^n}3 \text{ for all } i>i_n$ (cf. (\ref{eqn3.4})).
Then we can similarly construct another curve
$\bar\alpha_i(t)|_{t\in[0,\frac{|p_ip_{i_n}|}{2}]}\subset F$ with $\bar\alpha_i(0)=p_i$ such that
$|\Uparrow_p^{\bar\alpha_i(t)}\xi|<\delta^n \text{ for all } t\in[0, \frac{|p_ip_{i_n}|}{2}]$.
Note that $z_j$ can be chosen to ensure that  $\bar\alpha_i(t)|_{t\in[0,\frac{|p_ip_{i_n}|}{2}]}=\alpha_i(t)|_{t\in[0,\frac{|p_ip_{i_n}|}{2}]}$;
namely, for sufficiently large $i$ we have that
$$|\Uparrow_p^{\alpha_i(t)}\xi|<\delta^n \text{ for all } t\in[0, \frac{|p_ip_{i_n}|}{2}].$$
This implies that $|\Uparrow_p^{\alpha(t)}\xi|<\delta^n \text{ for all } t\in(0, \frac{|pp_{i_n}|}{2}]$,
i.e. $\alpha(t)$ is tangent to $\xi$ at $p$.
\end{proof}

\begin{remark}\label{rem3.7}{\rm
We would like to point out that (\ref{eqn3.8}) and (\ref{eqn3.9}) are inspired by the proof of Proposition 3.3 in \cite{PP1}.
Moreover, in the proof right above, $\alpha_i(t)$ is in fact the beginning part of the gradient curve of $\dist_p$ starting from $p_i$, 
and $\alpha(t)$ is just the beginning part of the radial curve starting from $p$ with direction $\xi$ (cf. Corollary \ref{coro4.2} below).}
\end{remark}


\section{Gradient curve and its applications}

\subsection{Gradient and radial curves}

In \cite{Pet}, gradient curve has been introduced for semi-concave functions,
distance functions in particular, on an $X\in \Alex(k)$. Let $p$ and $q$ be two distinct points in $X$, and let $\alpha(t)|_{t\in[0,+\infty)}$ be a continuous curve in $X$ with $\alpha(0)=q$.
Then $\alpha(t)|_{t\in[0,+\infty)}$ is the gradient curve of $\dist_p$ starting from $q$
if and only if, for each $t_0\in [0,+\infty)$, either $\Uparrow_{\alpha(t_0)}^{\alpha(t)}$ converges to  $\eta\in\Sigma_{\alpha(t_0)}X$
as $t\to t_0^+$ with
$$|\Uparrow_{\alpha(t_0)}^{p}\eta|=\max\{|\Uparrow_{\alpha(t_0)}^{p}\nu| : \nu\in \Sigma_{\alpha(t_0)}X\}>\frac\pi2$$
(more precisely, it is required that $\dist_p'(\alpha(t))|_{t=t_0^+}=-\cos|\Uparrow_{\alpha(t_0)}^{p}\eta|$), or $\alpha(t)\equiv\alpha(t_0)$ for all $t>t_0$ and there is no $\nu\in \Sigma_{\alpha(t_0)}X$ such that $|\Uparrow_{\alpha(t_0)}^{p}\nu|>\frac\pi2$.
Moreover, given $\xi\in\Sigma_pX$ and $\{p_i\}_{i=1}^\infty$ with $p_i\to p$ and $\Uparrow_p^{p_i}\to\xi$
as $i\to\infty$, we can define the radial curve starting from $p$ with direction $\xi$ by
the limit of gradient curves of $\dist_p$ starting from $p_i$ (\cite{Pet}).
Consequently, similar to `extremal' case (cf. Sections 3.1 and 4.1 in \cite{Pet}), we can see
the following two properties from Corollary C.

\begin{prop}\label{prop4.1}
Let $F$ be a subset in an $X\in \Alex(k)$. Then $F$ is quasi-convex if and only if, for any two distinct points $p,r\in F$, the gradient curve of $\dist_p$ starting from $r$ lies in $F$.
\end{prop}

\begin{coro}\label{coro4.2}
Let $F$ be a quasi-convex subset in an $X\in\Alex(k)$. Then, for any $p \in F$ and  any $\xi \in \Sigma_{p} F$, the radial curve starting from $p$ with direction $\xi$ lies in $F$.
\end{coro}

In Proposition \ref{prop4.1}, there is no need to assume that $F$ is closed in $X$; and for `extremal' case, $p$ can be an arbitrary point in $X$.
As another corollary of Proposition \ref{prop4.1}, we can see that the limit of quasi-convex subsets is also quasi-convex
(refer to Lemma 4.1.3 in \cite{Pet} for `extremal' case).

\begin{coro}\label{coro4.3}
Let $\{X_{n}\}_{n=1}^\infty$ be $m$-dimensional spaces in $\Alex(\kappa)$, and let $F_{n}$ be quasi-convex in $X_{n}$. If $X_{n} \stackrel{\text { GH }}{\longrightarrow} X$ with $F_{n} \rightarrow F \subset X$ as $n\to \infty$, then $F$ is also quasi-convex in $X$.
\end{coro}

\begin{proof} By Proposition \ref{prop4.1}, we just need to verify that, for any two distinct points $p,r\in F$, the gradient curve of $\dist_p$ starting from $r$ lies in $F$. Let $p_n, r_n\in F_n$ with $p_{n}\to p$ and $r_n\to r$ as $n\to \infty$. It is clear that $\dist_{p_n}$ converges to $\dist_p$ as $n\to\infty$. A fundamental fact is that the gradient curve
of $\dist_{p_n}$ starting from $r_n$ converges to the gradient curve
of $\dist_{p}$ starting from $r$ (Lemma 2.1.5 in \cite{Pet}). Due to the quasi-convexity of $F_n$, the gradient curve
of $\dist_{p_n}$ starting from $r_n$ lies in $F_n$, so the gradient curve of $\dist_p$ starting from $r$ lies in $F$.
\end{proof}

\begin{remark}\label{rem4.4}{\rm If $X_n$ converges to $X$ without collapse (in the Gromov-Hausdorff sense) in Corollary \ref{coro4.3}, then similar to `extremal' case (cf. Section 4.1 in \cite{Pet}) one can
show that $F_{n}$ also converges to $F$ with respect to induced intrinsic metrics from $X_n$ and $X$.}
\end{remark}

In the rest of this section, we will provide three applications of Proposition \ref{prop4.1}.

\subsection{Intersection of two quasi-convex subsets}

Via Proposition \ref{prop4.1}, we can see that the intersection of two quasi-convex subsets is also quasi-convex.

\begin{prop}\label{prop4.5}
Let $F$ and $G$ be two quasi-convex subsets in an $X\in\Alex(k)$. Then $F\cap G$ is also quasi-convex in $X$;
moreover, $\Sigma_{p} (F\cap  G)=\Sigma_{p} F\cap \Sigma_{p}G$ for any $p\in F\cap G$.
\end{prop}

\begin{proof} Since $F$ and $G$ are quasi-convex in $X$, for any $p,r\in F\cap G$ with $p\neq r$,
the gradient curve of $\dist_p$ starting from $r$ belongs to both $F$ and $G$ (by Proposition \ref{prop4.1}), which
implies that $F\cap G$ is also quasi-convex (by Proposition \ref{prop4.1} again).

Next, for any $p\in F\cap G$, we show that $\Sigma_{p} (F\cap  G)=\Sigma_{p} F\cap \Sigma_{p}G$.
It is obvious that $\Sigma_{p} (F\cap  G)\subseteq \Sigma_{p} F\cap \Sigma_{p}G$. On the other hand,
for any  $\xi\in \Sigma_{p} F\cap \Sigma_{p} G$, the radial curve starting from $p$ with direction $\xi$ lies in $F\cap G$
(by Corollary \ref{coro4.2}), which implies that $\Sigma_{p} (F\cap  G)\supseteq \Sigma_{p} F\cap \Sigma_{p}G$.
\end{proof}

\begin{remark}\label{rem4.6}{\rm It is true that the union of two extremal subsets is also extremal (\cite{PP1}).
However, in general, the union of two quasi-convex subsets might not be quasi-convex
(e.g., the union of two lines in a plane is not quasi-convex).}
\end{remark}

\begin{remark}\label{rem4.7}{\rm
Let $F$ and $G$ be two extremal subsets in an $X\in\Alex(k)$.
Without involving the concept of gradient curve,
\cite{PP1} has proven that both $F\cap G$ and $\overline{F\setminus G}$ are also
extremal in $X$ by showing $\Sigma_{p} (F\cap  G)=\Sigma_{p} F\cap \Sigma_{p}G$
and $\Sigma_{p} \overline{F\setminus G}=\overline{\Sigma_{p}F\setminus \Sigma_{p}G}$ firstly. However, we cannot give a proof for Proposition \ref{prop4.5} in such a way
(because the condition of `quasi-convex' is much weaker than `extremal').
Moreover, so far we cannot either show $\Sigma_{p} \overline{F\setminus G}=\overline{\Sigma_{p}F\setminus \Sigma_{p}G}$
or prove that $\overline{F\setminus G}$ is still quasi-convex if
$F$ and $G$ are quasi-convex in $X$ with $\dim(X)\geqslant 3$.}
\end{remark}

\subsection{Quasi-convex subsets in spherical suspensions}

Let $Z\triangleq\{z_1,z_2\}*Y$ with $|z_1z_2|=\pi$ and $Y\in \text{\rm Alex}(1)$ be a spherical suspension
(for details about such suspension structure refer to \cite{BGP}).
As examples of quasi-convex subsets, \cite{SSW} has shown that
if a quasi-convex subset $F$ in $Z$ contains at least two points including $z_1$, then $F=\{z_1,z_2\}*(F\cap Y)$.
In this paper, we provide a short proof and a stronger version of it via Proposition \ref{prop4.1}.

\begin{prop}\label{prop4.8}
Let $Z=\{z_1,z_2\}*Y$ with $|z_1z_2|=\pi$ and $Y\in \text{\rm Alex}(1)$,
and let $F$ be a quasi-convex subset in $Z$ containing at least two points. Then either $F\subseteq Y$, or there is $\bar z_1$, $\bar z_2$ and $\bar Y\in \text{\rm Alex}(1)$ with $|\bar z_1\bar z_2|=\pi$ such that $Z=\{\bar z_1,\bar z_2\}*\bar Y$ and $F=\{\bar z_1,\bar z_2\}*(F\cap \bar Y)$.
\end{prop}

\begin{proof} We first consider a special case where $z_1$ belongs to $F$. As mentioned above,
$F=\{z_1,z_2\}*(F\cap Y)$ in this case. In fact,
for any point $r\in F\setminus\{z_1\}$, $[rz_2]$ is the gradient curve of $\dist_{z_1}$ starting from $r$
by the spherical suspension structure of $Z$, and thus has to lie in $F$ by Proposition \ref{prop4.1}. In particular, $z_2\in F$. Similarly, if $r\neq z_2$, then $[rz_1]$ also belongs to $F$;
namely, the minimal geodesic $[z_1z_2]$ passing $r$ belongs to $F$. This implies that $F=\{z_1,z_2\}*(F\cap Y)$.

We now can assume that $z_1\not\in F$, $z_2\not\in F$ and $F\cap(Z\setminus Y)\neq\emptyset$.
Then we can assume that there is $\bar z_1\in F$ and $\bar z_2\in F$ such that $\frac\pi2>|z_1F|=|z_1\bar z_1|\leqslant|z_2\bar z_2|=|z_2F|$. So, if $|\bar z_1\bar z_2|<\pi$, then the spherical suspension structure of $Z$ guarantees that,  for any $[\bar z_1\bar z_2]$,
$$|\uparrow_{\bar z_2}^{\bar z_1}\uparrow_{\bar z_2}^{z_1}|<\frac\pi2
\text{ and } |\uparrow_{\bar z_2}^{\bar z_1}\uparrow_{\bar z_2}^{z_2}|=\pi-|\uparrow_{\bar z_2}^{\bar z_1}\uparrow_{\bar z_2}^{z_1}|.$$
However, the quasi-convexity of $F$ implies that $|\uparrow_{\bar z_2}^{\bar z_1}\uparrow_{\bar z_2}^{z_2}|\leqslant\frac\pi2$, a contradiction. Namely, it has to hold that $|\bar z_1\bar z_2|=\pi$ (so $z_i$ and $\bar z_i$ lie in a (great) circle of perimeter $2\pi$). Hence, there is $\bar Y\in \text{\rm Alex}(1)$ such that $Z=\{\bar z_1,\bar z_2\}*\bar Y$; and thus, similar to the special case above, $F=\{\bar z_1,\bar z_2\}*(F\cap \bar Y)$.
\end{proof}

\subsection{Fixed point set of an isometry}

\begin{prop}\label{prop4.9}
Let $X\in \text{\rm Alex}(k)$, and let $F$ be the fixed point set of an isometry on $X$. Then $F$ is quasi-convex in $X$.
\end{prop}

Recall that the fixed point set of an isometry on a complete Riemannian manifold is totally geodesic, while a quasi-convex subset in a complete Riemannian manifold must be totally geodesic (\cite{SSW}).

\begin{proof} We need only to consider the case where $F$ contains at least two points.
Let $p$ and $r$ be arbitrary two distinct points in $F$. By the uniqueness of the gradient curve of $\dist_p$ starting from $r$, it must be fixed by the isometry. I.e., the gradient curve of $\dist_p$ starting from $r$
belongs to $F$, so $F$ is quasi-convex by Proposition \ref{prop4.1}.
\end{proof}

\begin{remark}\label{rem4.10}{\rm
Let $\gamma$ be the isometry fixing $F$ in Proposition \ref{prop4.9}, and let $p\in F$. Note that there is a naturally induced isometry $\bar \gamma$ on $\Sigma_pX$, and $\Sigma_pF$ belongs to the fixed point set $\bar F$ of $\bar\gamma$.
On the other hand, by the uniqueness of radial curve starting from a point with a fixed direction,
$\gamma$ must fix the radial curve starting from $p$ with any direction $\xi\in\bar F$. Namely, $\Sigma_pF=\bar F$.}
\end{remark}

\begin{remark}\label{rem4.11}{\rm
Let $X\in \text{\rm Alex}(k)$, and let $\Gamma$ be a compact group which acts on $X$ by isometries
with nonempty fixed point set $F$. In \cite{PP1}, it has been shown that $F$
is extremal as a subset of the orbit space $X/\Gamma$ (where a key tool is `strictly convex hull'). Based on this, \cite{SSW} has proven that $F$ is quasi-convex in $X$. Note that we can prove it using the same arguments as in the proof of Proposition \ref{prop4.9}. (We would like to point out that, using the technique of strictly convex hull, one can also see that $\Sigma_pF=\bar F$ in Remark \ref{rem4.10}.)}
\end{remark}


\noindent School of Mathematical Sciences (and Lab. math. Com.
Sys.), Beijing Normal University, Beijing, 100875
P.R.C.\\
e-mail: suxiaole$@$bnu.edu.cn; wyusheng$@$bnu.edu.cn

\vskip2mm

\noindent Mathematics Department, Capital Normal University,
Beijing, 100037 P.R.C.\\
e-mail: 5598@cnu.edu.cn

\end{document}